\RequirePackage{fix-cm}
\RequirePackage{amsmath}
\documentclass[smallextended]{svjour3}       
\smartqed  
\usepackage{graphicx}
 \usepackage{txfonts,bm} 
\usepackage{dsfont}
\usepackage{amssymb}
\usepackage{enumerate}
\usepackage{enumitem}
\usepackage{dsfont}
\usepackage{rotating}
\usepackage{cancel}
\usepackage{extarrows}
\usepackage{xfrac}
\usepackage{multirow}
\usepackage{caption}
\usepackage{subcaption}
\usepackage{algorithm}
\usepackage{algpseudocode}
\usepackage{url}
\usepackage{amsopn}
\usepackage{natbib}
  \def\citeapos#1{\citeauthor{#1}'s (\citeyear{#1})}
  
  \bibpunct[, ]{(}{)}{;}{a}{}{,}

\RequirePackage[OT1]{fontenc}
\RequirePackage[colorlinks,citecolor=blue,urlcolor=blue,linkcolor=blue]{hyperref}
\usepackage{booktabs,multirow}
\usepackage{rotating}
\usepackage{environ}
\newcommand{\RR}{\mathds{R}}
\newcommand{\NN}{\mathds{N}}
\newcommand{\ds}{\displaystyle}
\newcommand{\E}{\operatorname{\mathds{E}}}
\newcommand{\Var}{\operatorname{\mathsf{Var}}}
\newcommand{\card}{\operatorname{card}}
\newcommand{\tr}{\operatorname{tr}}
\newcommand{\diag}{\operatorname{diag}}
\newcommand{\rank}{\operatorname{rank}}
\newcommand{\du}{\operatorname{d}}
\newcommand{\ba}{\bm{\alpha}}
\newcommand{\bb}{\bm{\beta}}
\newcommand{\bg}{\bm{\gamma}}
\newcommand{\ud}{{\rm d}}
\newcommand{\ta}{T}
\newcommand{\approxn}{\stackrel{n}{\approx}}
\newcommand{\law}{\xlongequal{\textrm{\rm d}}}
\newcommand{\h}[1]{\hat{#1}}
 \journalname{TEST}
\let\originalleft\left
\let\originalright\right
\renewcommand{\left}{\mathopen{}\mathclose\bgroup\originalleft}
\renewcommand{\right}{\aftergroup\egroup\originalright}
\begin{document}

\title{Uniform Integrability of the OLS Estimators, and the Convergence of their Moments
       \thanks{Both authors acknowledge support provided by the Department of Biostatistics and the Jacobs School of Medicine and Biomedical Sciences, University at Buffalo (in the form of start-up package to the second author). In addition, the authors would like to thank the editor and reviewers of this manuscript for suggestions that improved the quality of the paper.}
      }


\author{Georgios Afendras \and
        Marianthi Markatou
}


\institute{G. Afendras \at
              Department of Biostatistics and Jacobs School of Medicine and Biomedical Sciences, University at Buffalo, 811 Kimball Tower, Buffalo, NY 14214, USA \\
              Tel.: +1 (716) 829-5094 \\
              Fax: +1 (716) 829-2200 \\
              \email{gafendra@buffalo.edu}           
           \and
           M. Markatou \at
              Department of Biostatistics and Jacobs School of Medicine and Biomedical Sciences, University at Buffalo, 726 Kimball Tower, Buffalo, NY 14214, USA \\
              \email{markatou@buffalo.edu}
}

\date{}

\maketitle
\vspace{-.1ex}

\begin{abstract}
The problem of convergence of moments of a sequence of random variables to the moments of its asymptotic distribution is important in many applications. These include the determination of the optimal training sample size in the cross validation estimation of the generalization error of computer algorithms, and in the construction of graphical methods for studying dependence patterns between two biomarkers. In this paper we prove the uniform integrability of the ordinary least squares estimators of a linear regression model, under suitable assumptions on the design matrix and the moments of the errors. Further, we prove the convergence of the moments of the estimators to the corresponding moments of their asymptotic distribution, and study the rate of the moment convergence. The canonical central limit theorem corresponds to the simplest linear regression model. We investigate the rate of the moment convergence in canonical central limit theorem proving a sharp improvement of \citeauthor{Bahr1965}'s (\citeyear{Bahr1965}) theorem.
\end{abstract}

\section{Introduction}
\label{sec:intro}

Regression models play a central role in statistics, for prediction and statistical inference. The most famous and, because of its extensive use, important model is the linear regression model,
\begin{equation}
\label{eq.model}
\bm{Y}=\mathbf{X}\bm{\beta}+\bm{\varepsilon},
\end{equation}
where $\bm{Y}=(Y_1,\ldots,Y_n)^\ta$ is the vector of response variables, $\mathbf{X}=(\bm{x}_1,\ldots,\bm{x}_n)^\ta\in\RR^{n\times p}$ is the design matrix,
$\bm{\beta}=(\beta_0,\beta_1,\ldots,\beta_{p-1})^\ta\in\RR^p$ is the parameter vector and $\bm{\varepsilon}=(\varepsilon_1,\ldots,\varepsilon_n)^\ta$ is the vector of the errors, where the $\varepsilon_i$'s are independent and $\varepsilon_i$ has $\E(\varepsilon_i)=0$, $\Var(\varepsilon_i)=\sigma^2$, and distribution function $F_i$. If $\rank(\mathbf{X})=p$, the Ordinary Least Squares estimator of $\bb$ is
\begin{equation}
\label{eq.OLS}
\bm{b}=(\mathbf{X}^\ta\mathbf{X})^{-1}\mathbf{X}^{\ta}\bm{Y}.
\end{equation}

Let $\mathbf{M}_n=(m_{ij;n})\doteq \mathbf{X}^\ta \mathbf{X}\in\RR^{p\times p}$ and $\mathbf{D}_n\doteq \diag(\sqrt{m_{11;n}},\ldots,\sqrt{m_{pp;n}})$. Suppose
{\rm(i)} $m_{ii;n}\to \infty$ as $n\to\infty$, $i=1,\ldots,p$,
{\rm(ii)} $x_{i;n+1}^2/m_{ii;n}\to0$ as $n\to\infty$, $i=1,\ldots,p$,
{\rm(iii)} $\mathbf{D}_n^{-1}\mathbf{M}_n \mathbf{D}_n^{-1}\to \mathbf{V}^{-1}$ as $n\to\infty$, where $\mathbf{V}^{-1}$ is nonsingular, and
{\rm(iv)} $\sup_{i=1,2,\ldots}\int_{|u|>c}u^2\ud F_i(u)\to 0$ as $c\to\infty$ [if $\varepsilon_i$'s are identically distributed, this assumption follows]. Then, $\mathbf{D_n}(\bm{b}-\bm{\beta})\xlongrightarrow{\rm d}N_p(\bm{0},\sigma^2\mathbf{V})$ as $n\to\infty$, see \citet[][Sec.~2.6]{Anderson1971}. Note that the most general normalization is described in \citet[][Sec.~1.11]{Mynbaev2011}. As a special case of the aforementioned conditions and the associated normality result we discuss the following situation. Suppose {\rm(iii$'$)} $n^{-1}\mathbf{X}^\ta \mathbf{X}\to \mathbf{V}^{-1}$ as $n\to\infty$, where $\mathbf{V}^{-1}\in\RR^{p\times p}$ is nonsingular. Then, Assumptions {\rm(i)}, {\rm(ii)} are satisfied; and Assumption {\rm(iii)} is satisfied with $\mathbf{\Delta}^{-1}\mathbf{V}^{-1}\mathbf{\Delta}^{-1}$ instead of $\mathbf{V}^{-1}$, where $\mathbf{\Delta}$ is the diagonal $p\times p$ matrix with elements the square roots of the diagonal elements of $\mathbf{V}^{-1}$. Writing  $\sqrt{n}(\bm{b}-\bm{\beta})\equiv\sqrt{n}\mathbf{I}_p(\bm{b}-\bm{\beta})=\sqrt{n}\mathbf{D_n}^{-1}\cdot \mathbf{D_n}(\bm{b}-\bm{\beta})$, where $\mathbf{I}_p$ is the $p\times p$ identity matrix, an application of Slutsky's theorem gives $\sqrt{n}(\bm{b}-\bm{\beta})\xlongrightarrow{\rm d}N_p(\bm{0},\sigma^2\mathbf{V})$ as $n\to\infty$. Therefore, under {\rm(i)}--{\rm(iv)} or {\rm(iii$'$)}, {\rm(iv)},
\begin{equation}
\label{eq.N_p}
{\mathbf{\Gamma}_n}(\bm{b}-\bm{\beta})\xlongrightarrow{\rm d}N_p(\bm{0},\sigma^2\mathbf{V})
\ \ \textrm{as}\ \ n\to\infty,
\ \ \ \textrm{where}\ \ {\mathbf{\Gamma}_n}=\mathbf{D_n} \ \textrm{or} \ \sqrt{n}\mathbf{I}_p \ \textrm{respectively}.
\end{equation}

Let $\{Y_n\}$ be an independent and identically distributed (iid) sequence from a distribution $F$ with mean $\mu$ and variance $0<\sigma^2<\infty$. The canonical Central Limit Theorem (CLT) says that, the sequence of the standardized averages, $Z_n=\sqrt{n}\frac{\bar{Y}-\mu}{\sigma}$, where $\bar{Y}=\frac{1}{n}\sum_{j=1}^{n}Y_j$, converges in distribution to the standard normal random variable $Z$. This result can be represented as the simplest linear regression model $\bm{Y}=\bm{1}\mu+\bm{\varepsilon}$, where $\bm{1}=(1,\ldots,1)^\ta\in\RR^n$ (one can easily see that assumption {\rm(iii$'$)} is satisfied), since the ordinary least squares (OLS) estimator of $\mu$ is $\h{\mu}=\bar{Y}$. The rate of convergence of the absolute moments of $Z_n$ to the absolute moments of $Z$ (of order $r$, for specific positive real numbers of $r$) has been studied by various authors \citep[see][among others]{BE1965,BR1976,Hall1978,Hall1982}. \citet{Bahr1965} addressed the problem of the convergence of the moments of $Z_n$ to the corresponding moments of $Z$, and provided their rate of convergence. Specifically, he proved that, if $\E|Y_i|^r<\infty$ for some positive integer $r$, then
\begin{equation}
\label{eq.rate_Bahr}
\E(Z_n^r)-\E(Z^r)=O(n^{-1/2}).
\end{equation}

\citet{AM2016} study the problem of selecting the optimal value of the training sample size in the cross validation estimation of the generalization error of computer algorithms. When the decision rule is given by a linear regression model and the loss function is the squared error loss, their analysis requires to know the cases in which the moments of the OLS estimators converge to the corresponding moments of their asymptotic distribution, thereby giving rise to the work presented here. 

This paper is organized as follows. Section \ref{sec:OLS} proves the uniform integrability of the OLS estimators, under a natural condition on the moments of the errors. In Section \ref{sec:rate} we improve \citeauthor{Bahr1965}'s Theorem; and, in general, we study the rate of convergence of the moments of the OLS estimators to the corresponding moments of their asymptotic distribution. Specifically, we study the rate of convergence of the second and third moments of the OLS estimators. We obtain that, the rate of convergence depends on the power of the moment (even or odd) as well as the design matrix $\mathbf{X}$. Section \ref{sec.disc} concludes with a short summary and discussion of the results.

\section{Uniform Integrability of the OLS estimators}
\label{sec:OLS}
In view of \eqref{eq.N_p}, a natural question arises: {\it When do the moments of ${\mathbf{\Gamma}_n}(\bm{b}-\bm{\beta})$ converge, as $n\to\infty$, to the corresponding moments of its asymptotic distribution?}

Let $\ba\in\RR^p$ and let us consider the sequence of random variables (rv's) $\xi_n\doteq\ba^\ta {\mathbf{\Gamma}_n}(\bm{b}-\bm{\beta})$. From \eqref{eq.N_p}, using {\it delta}-method \citep[see][p.~25]{Vaart1998}, we have that
\begin{equation}
\label{eq.N-xi}
\xi_n\xlongrightarrow{\rm d}\xi \ \ \textrm{as} \ \ n\to\infty,
\quad\textrm{where} \ \xi\sim N(0,\sigma^2\ba^\ta\mathbf{V} \ba).
\end{equation}
We want $\E|\xi_n|^r\to\E|\xi|^r$ as $n\to\infty$, for some $r>0$. Obviously, if $\ba=\bm{0}$, the result is trivial; thus, we study the nontrivial cases in which $\ba\ne\bm{0}$. From \eqref{eq.OLS}, $\xi_n=\sum_{i=1}^{n}\kappa_i\varepsilon_i$, where $\kappa_i=\ba^{\ta}{\mathbf{\Gamma}_n}(\mathbf{X}^\ta \mathbf{X})^{-1}\bm{x}_i$. This relation shows that it is required $\E|\varepsilon_i|^r<\infty$. On the other hand the relation $\E|\xi_n|^r<\infty$ does not imply that $\E|\xi_n|^r\to\E|\xi|^r$ as $n\to\infty$. The desired convergence of the moments of $\xi_n$ is true if $\{|\xi_n|^r\}$ is uniformly integrable \citep[see, e.g.,][Theo.~25.12, p.~338; cf \citeauthor{Loeve1977} \citeyear{Loeve1977}, Sec.~$^*$11.4, pp.~184--187]{Billingsley1995}. Proving the uniform integrability of the OLS estimators, answers the aforementioned question.

\begin{definition}
\label{def.LM-even}
Let $R>0$ be fixed. We define $\mathcal{LM}(R)$ to be the family of models containing models of the form \eqref{eq.model} such that {\rm(i)}--{\rm(iv)} or {\rm(iii$'$)}, {\rm(iv)} hold and $\epsilon^R<\infty$, where $\epsilon^R\doteq\sup_{i=1,2,\ldots}\big\{\E|\varepsilon_i|^R\big\}$.
\end{definition}

\begin{theorem}
\label{theo.UI1}
Let $R\ge2$ and the model belongs to $\mathcal{LM}(R)$. Then for each $0<r<R$ the sequence of rv's $|\xi_n|^r$ in \eqref{eq.N-xi} is uniformly integrable; and thus,
\[
\E|\xi_n|^r\to\E|\xi|^r,
\ \ \textrm{and, if $r$ is a positive integer} \ \
\E(\xi_n^r)\to\E(\xi^r) \ \ \textrm{as} \ \ n\to\infty.
\]
\end{theorem}
\begin{proof}
For $0<r<2$, set $\delta=(2-r)/r>0$. So,
$\E\left(|\xi_n|^r\right)^{1+\delta}=\E|\xi_n|^2
=\E\big(
\sum_{i=1}^{n}\kappa_i^2\varepsilon_i^2
+2\mathop{\sum\sum}_{1\le{i}<j\le{n}}\kappa_i\kappa_j\varepsilon_i\varepsilon_j
\big)
=\sigma^2\sum_{i=1}^{n}\kappa_i^2$.
Set $\bg^T=\ba^T {\mathbf{\Gamma}_n}$. Using Cauchy-Schwarz inequality, we get $\kappa_i^2=\big(\bg^{\ta}(\mathbf{X}^\ta \mathbf{X})^{-1}\bm{x}_i\big)^2\le \big(\bg^{\ta}(\mathbf{X}^\ta \mathbf{X})^{-1}\bg\big)\big(\bm{x}_i^{\ta}(\mathbf{X}^\ta \mathbf{X})^{-1}\bm{x}_i\big)$. Using the relation $\sum_{i=1}^{n}h_{ii}=\tr(\mathbf{H})=p$, where $\mathbf{H}=\mathbf{X}(\mathbf{X}^\ta \mathbf{X})^{-1}\mathbf{X}^\ta$ is the hat matrix of the model \citep[see, e.g.,][]{CH1998},
$
\sum_{i=1}^{n}\kappa_i^2\le \ba^{\ta}\big({\mathbf{\Gamma}_n} \mathbf{M}_n^{-1} {\mathbf{\Gamma}_n}\big)\ba\sum_{i=1}^{n}h_{ii}=p\ba^{\ta}\big({\mathbf{\Gamma}_n} \mathbf{M}_n^{-1} {\mathbf{\Gamma}_n}\big)\ba
$.
Thus, $\E\left(|\xi_n|^r\right)^{1+\delta}\le p\sigma^2\ba^{\ta}\big({\mathbf{\Gamma}_n} \mathbf{M}_n^{-1} {\mathbf{\Gamma}_n}\big)\ba$.
Assumption {\rm(iii)} or {\rm(iii$'$)} gives $p\sigma^2\ba^{\ta}\big({\mathbf{\Gamma}_n} \mathbf{M}_n^{-1} {\mathbf{\Gamma}_n}\big)\ba \to p\sigma^2\ba^{\ta}\mathbf{V}\ba<\infty$. So, $p\sigma^2\ba^{\ta}\big({\mathbf{\Gamma}_n} \mathbf{M}_n^{-1} {\mathbf{\Gamma}_n}\big)\ba$ is a bounded sequence, which imply that $\E\left(|\xi_n|^r\right)^{1+\delta}$ is also bounded; and hence, $\sup_n\E\left(|\xi_n|^r\right)^{1+\delta}<\infty$. Thus, $\{|\xi_n|^r\}$ is uniformly integrable \citep[see, e.g.,][p.~338]{Billingsley1995}.

For $2\le r<R$ (only if $R>2$), set $\delta=(R-r)/r$; so, $\E\left(|\xi_n|^r\right)^{1+\delta}=\E|\xi_n|^R$. An application of Marcinkiewicz-Zygmund inequality gives
$
\E|\xi_n|^R\le C_R\E\left(\sum_{i=1}^{n}\kappa_i^2\varepsilon_i^2\right)^{R/2}
$,
where $C_R>0$ is a positive constant which depends only on $R$. By using Minkowski inequality, since $R/2>1$, it follows that
$
\E\left(\sum_{i=1}^{n}\kappa_i^2\varepsilon_i^2\right)^{R/2}
\le \left(\sum_{i=1}^{n}\kappa_i^2\left(\E|\varepsilon_i|^R\right)^{2/R}\right)^{R/2}
\le\epsilon^R\left(\sum_{i=1}^{n}\kappa_i^2\right)^{R/2}
\le \left(p\ba^{\ta}\big({\mathbf{\Gamma}_n} \mathbf{M}_n^{-1} {\mathbf{\Gamma}_n}\big)\ba\right)^{R/2}\epsilon^R
$.
Thus, the following relation arises\linebreak
$
\E|\xi_n|^R\le C_R \left(p\ba^{\ta}\big({\mathbf{\Gamma}_n} \mathbf{M}_n^{-1} {\mathbf{\Gamma}_n}\big)\ba\right)^{R/2}\epsilon^R
\to C_R \left(p\ba^{\ta}\mathbf{V}\ba\right)^{R/2}\epsilon^R<\infty
$.
As above, we have that the sequence $|\xi_n|^r$ is uniformly integrable.

Finally, if $r$ is a positive integer, the sequence $\xi_n^r$ is well defined and the fact that $|\xi_n^r|=|\xi_n|^r$ completes the proof.
\hfill$\square$
\end{proof}

Next, we generalize Theorem \ref{theo.UI1} to the $k$-dimensional case. Let $\ba_j\in\RR^p\smallsetminus\{\bm{0}\}$, $j=1,\ldots,k$, and let us consider the random vector $\bm{\xi}_n$ with components $\xi_{n,j}=\ba_j^\ta {\mathbf{\Gamma}_n}(\bm{b}-\bm{\beta})$, $j=1,\ldots,k$; that is,
$
\bm{\xi}_n\doteq \mathbf{A}{\mathbf{\Gamma}_n}(\bm{b}-\bm{\beta})
$,
$\mathbf{A}=(\ba_1,\ldots,\ba_k)^\ta$. Using {\it delta}-method, \eqref{eq.N_p} gives
\begin{equation}
\label{eq.N-xi-k}
\bm{\xi}_n\xlongrightarrow{\rm d}\bm{\xi} \ \ \textrm{as} \ \ n\to\infty,
\quad\textrm{where} \ \ \bm{\xi}\sim N_k(\bm{0},\sigma^2\mathbf{A}\mathbf{V}\mathbf{A}^\ta).
\end{equation}

\begin{theorem}
\label{theo.UI2}
Let $R\ge2$ and the model belongs to $\mathcal{LM}(R)$. Then, for each $r_1,\ldots,r_k\ge0$ such that $r_1+\cdots+r_k<R$ the sequence of rv's $\prod_{j=1}^k|\xi_{n,j}|^{r_j}$ is uniformly integrable [$\xi_{n,j}$ as in \eqref{eq.N-xi-k}]; and thus,
\[
\E\prod\limits_{j=1}^k|\xi_{n,j}|^{r_j}\to\E\prod\limits_{j=1}^k|\xi_{j}|^{r_j} \ \ \textrm{as} \ n\to\infty,
\]
where $\xi_j$ are the components of $\bm{\xi}$ in \eqref{eq.N-xi-k}; and if $r_1,\ldots,r_k$ are nonnegative integers
\[
\E\prod\limits_{j=1}^k\xi_{n,j}^{r_j}\to\E\prod\limits_{j=1}^k\xi_{j}^{r_j} \ \ \textrm{as} \ n\to\infty.
\]
\end{theorem}
\begin{proof}
Let $r_{\max}=\max_{j=1,\ldots,k}\{r_j\}$. If $r_{\max}=0$, the result is obvious. For the case $r_{\max}>0$, set $r=r_1+\cdots+r_k<R$, $p_j=\frac{R}{r_j+\frac{R-r}{k}}$, $j=1,\ldots,k$, and $\delta=\frac{R-r}{2kr_{\max}}>0$. Observe that $p_j>1$ for all $j=1,\ldots,k$ and $\sum_{j=1}^{k}1/p_j=1$. Thus, the generalized H\"older inequality, see \citet{Cheung2001}, gives
$
\E\left|\prod_{j=1}^{k}|\xi_{n,j}|^{r_j}\right|^{1+\delta}\le \prod_{j=1}^{k}\left(\E\left|\xi_{n,j}\right|^{(1+\delta)r_jp_j}\right)^{1/p_j}
$.
Also, observe that $0<(1+\delta)r_jp_j=\frac{1+(R-r)/(2kr_{\max})}{1+(R-r)/(kr_j)}R<R$ for each $j=1,\ldots,k$; thus, Theorem \ref{theo.UI1} shows that $\lim_{n\to\infty}\E\left|\xi_{n,j}\right|^{(1+\delta)r_jp_j}=\E|\xi_j|^{(1+\delta)r_jp_j}<\infty$, since $\xi_j\sim N(0,\sigma^2\ba_j^\ta \mathbf{V} \ba_j)$. So,
$
\prod_{j=1}^{k}\left(\E\left|\xi_{n,j}\right|^{(1+\delta)r_jp_j}\right)^{1/p_j}\to
\prod_{j=1}^{k}\left(\E\left|\xi_{j}\right|^{(1+\delta)r_jp_j}\right)^{1/p_j}<\infty
$
as $n \to \infty$, and using the same arguments as in proof of Theorem \ref{theo.UI1}, the proof is completed.
\hfill$\square$
\end{proof}

\begin{remark}
Theorem \ref{theo.UI2} is important because it generalizes the classical problem of the convergence of moments in the canonical CLT; also, it provides an answer for the approximation of the moments of the OLS estimators in the linear regression analysis, when this is needed by the researchers \citep[see, e.g.,][]{AM2016}.
\end{remark}

\section{On the rate of convergence}
\label{sec:rate}

In this section we study the rate of convergence of the moment convergence that is presented in Theorem \ref{theo.UI1}.

\subsection{Canonical CLT case}
\label{ssec:rate.CLT}

Initially, we note that \citeapos{Bahr1965} results are obtained using the properties of the Fourier-Stieltjes transform \citep[see][Lemmas~1--3]{Bahr1965} of a function. These results apply to both integer type powers $k$, even and odd \citep[see][Theorem~1]{Bahr1965}, of the standardized average. This method of proof allows one to obtain only the weaker result, that is, the root $n$ rate of convergence. A relevant to moment convergence in CLT work is presented by \citet{Brown1970}. We use a combinatorial method of proof to first obtain exact expressions of the moments of the standardized average. Then, based on these expressions, further combinatorial arguments allow one to study the rate of convergence of moments separately for even and odd powers $k$. This is done because it is unclear how the \citeauthor{Bahr1965}'s method can be modified to apply, even if exact expressions of the even and odd moments are known.

Let $Y_n$ be an iid collection of rv's from a distribution $F$ with mean $\mu$ and variance $0<\sigma^2<\infty$. The sequence of the standardized averages $Z_n\doteq\frac{S_n}{\sqrt{n}}$, where $S_n\doteq\sum_{j=1}^{n}\frac{Y_j-\mu}{\sigma}\doteq\sum_{j=1}^{n}W_j$, converges in law to the standard normal random variable $Z$. Assume that $\E|Y_j|^r$ is finite for some integer $r\ge2$. It is well known that $S_n^r=\sum_{{j_1,\ldots,j_n\in\NN, j_1+\cdots+j_n=r}}{r\choose j_1,\ldots,j_n}\prod_{i=1}^{n}W_i^{j_i}$, where ${r\choose j_1,\ldots,j_n}=\frac{r!}{j_1!\cdots j_n!}$. Thus, $\E(S_n^r)=\sum_{{j_1,\ldots,j_n\in\NN, j_1+\cdots+j_n=r}}{r\choose j_1,\ldots,j_n}\prod_{i=1}^{n}\E\left(W_i^{j_i}\right)$  due to independence of the $Y$'s. In a product $\prod_{i=1}^{n}\E\left(W_i^{j_i}\right)$, if $j_{i_0}=0$ for some index $i_0$, then this index does not affect both the product and its multinomial coefficient, and thus, it can be omitted. Because the $Y_i$s are identically distributed, each $\prod_{i=1}^{n}\E\left(W_i^{j_i}\right)$ is of the form $\prod_{i=1}^{m}\E\left(W_i^{j_i}\right)$ for $1\le m\le \min\{r,n\}$ and $2\le j_1\le\cdots\le j_m\le r$ such that $j_1+\cdots+j_m=r$ (note that we are interested in the case $n\to\infty$, so assume that $n>r$). Under the above observations, we define
\[
\bm{J}(r)\doteq
\left\{
\bm{j}_m\mid m\in\{1,\ldots,\lfloor r/2\rfloor\}, \ \bm{j}_m=(j_1,\ldots,j_m)^\ta\in\NN^m \colon {\ds 2\le j_1\le\cdots\le j_m, \atop \ds j_1+\cdots+j_m=r}
\right\},
\]
where $\lfloor\cdot\rfloor$ denotes the integer part function. It is obvious that $\card(\bm{J}(r))<\infty$. For each $\bm{j}_m\in\bm{J}(r)$ we denote
\[
E_{\bm{j}_m}\doteq\prod_{i=1}^{m}\E\left(W_i^{j_i}\right)\in\RR.
\]
Therefore,
\[
\E(S_n^r)=\sum_{\bm{j}_m\in\bm{J}(r)}c_{\bm{j}_m}(r)E_{\bm{j}_m}.
\]
The number $m$ is called {\it length} of $\bm{j}_m$, and the coefficient $c_{\bm{j}_m}(r)$ which corresponds to the maximum length is called {\it leading coefficient}. For the computation of $c_{\bm{j}_m}(r)$ we define $j^*_1<\ldots<j^*_{m^*}$ to be the distinct powers of $j_1\le\cdots\le j_m$, where $m^*\in\{1,\ldots,m\}$ is the number of the distinct powers, and
\[
d_k\doteq \{\textrm{the multiplicity of} \ j^*_k \ \textrm{in} \ \bm{j}_m\},
\quad
k=1,\ldots,m^*.
\]
Using standard combinatoric arguments, for $n>r$,
\begin{equation}
\label{eq.cjmr}
c_{\bm{j}_m}(r)=\frac{{r\choose \bm{j}_m}}{d_1!\cdots d_{m^*}!}(n)_m,
\quad\textrm{where} \ \ {r\choose \bm{j}_m}=\frac{r!}{j_1!\cdots j_m!} \ \ \textrm{and} \ \ (n)_m=\frac{n!}{(n-m)!}.
\end{equation}
Now we are in a position to state and prove the following theorem, which improves \citeauthor{Bahr1965}'s Theorem, which is stated in relation \eqref{eq.rate_Bahr}.

\begin{theorem}
\label{theo.CLT}
Let $Y_n$ be an iid sequence of rv's from a distribution $F$ with mean $\mu$ and variance $0<\sigma^2<\infty$. If $\E|Y|^R<\infty$ for some integer $R\ge2$, then for any positive integer $r\le R$,
\[
\E(Z_n^r)-\E(Z^r)
=\left\{
\begin{array}{c@{,\quad if~}l}
  O(n^{-1/2}) & r \ \textrm{is odd}, \\
  O(n^{-1}) & r \ \textrm{is even},
\end{array}
\right.
\]
where $Z_n$ is the sequence of the standardized averages of $Y$'s and $Z$ is the standard normal random variable. Specifically, if $r=2k+1\le R$,
\begin{equation}
\label{eq.lim.odd}
\lim_{n\to\infty} \sqrt{n}\E(Z_n^{2k+1})=\frac{k(2k+1)(2k-1)!!}{3\sigma^3}\mu_3,
\end{equation}
and if $r=2k\le R$,
\begin{equation}
\label{eq.lim.even}
\lim_{n\to\infty} n\left(\E(Z_n^{2k})-(2k-1)!!\right)=(2k-1)!!\left(\frac{(k)_2\mu_4}{6\sigma^4}+\frac{(k)_3\mu_3^2}{9\sigma^6}-\frac{(k)_2}{2}\right);
\end{equation}
where $(2k-1)!!=1\times3\times\cdots\times(2k-1)$ is the double factorial function and $\mu_\nu$ denotes the $\nu$-th central moment of $Y$'s.
\end{theorem}

\begin{proof}
For the cases $r=1,2$ the results are obvious. It remains to prove the cases $r\ge3$ for odd moments and $r\ge4$ for even moments. Observe that $\E(Z_n^r)=n^{-r/2}\E(S_n^r)$.

If $r$ is odd, $3\le r=2k+1\le R$, the leading coefficient corresponds to $m=k$, $\bm{j}_{k}=(2,\ldots,2,3)^\ta\in\NN^k$, $m^*=2$, $j_1^*=2$, $j_2^*=3$, $d_1=k-1$, $d_2=1$, $E_{\bm{j}_{k}}=\mu_3/\sigma^3$. So,
$
c_{\bm{j}_{k}}(2k+1)=\frac{k(2k+1)(2k-1)!!}{3}(n)_k=O(n^k)
$,
In view of \eqref{eq.cjmr} and since $(n)_m\approxn n^m$, $c_{\bm{j}_{m}}(2k+1)=o(n^k)$ for all $m<k$; and thus,
$
\E(Z_n^{2k+1})=\frac{\E(S_n^{2k+1})}{n^k n^{1/2}}
=\frac{c_{\bm{j}_{k}}(2k+1)\mu_3/\sigma^3+o(n^k)}{n^k n^{1/2}}
=\frac{O(n^k)}{n^k n^{1/2}}
=O(n^{-1/2})
$.

If $r$ is even, $4\le r=2k\le R$, the leading coefficient corresponds to
$m=k$, $\bm{j}_{k}=(2,\ldots,2)^\ta\in\NN^k$, $m^*=1$, $j_1^*=2$, $d_1=k$, $E_{\bm{j}_{k}}=1$. So,
$
c_{\bm{j}_{k}}(2k)=(2k-1)!!(n)_k
$.
The second leading coefficients, the coefficients with length $m=k-1$, correspond to
$\bm{j}_{k-1}^{(1)}=(2,\ldots,2,4)^\ta\in\NN^{k-1}$, $m^*=2$, $j_1^*=2$, $j_2^*=4$, $d_1=k-2$, $d_2=1$,  $E_{\bm{j}_{k}^{(1)}}=\mu_4/\sigma^4$; and to $\bm{j}_{k-1}^{(2)}=(2,\ldots,2,3,3)^\ta\in\NN^{k-1}$, $m^*=2$, $j_1^*=2$, $j_2^*=3$, $d_1=k-3$, $d_2=2$,  $E_{\bm{j}_{k}^{(2)}}=\mu_3^2/\sigma^6$. Thus,
$
c_{\bm{j}_{k-1}^{(1)}}(2k)=\frac{(k)_2(2k-1)!!}{6}(n)_{k-1}
$,
$
c_{\bm{j}_{k-1}^{(2)}}(2k)=\frac{(k)_3(2k-1)!!}{9}(n)_{k-1}
$.
In view of \eqref{eq.cjmr} and since $(n)_m\approxn n^m$, $c_{\bm{j}_{m}}(2k)=o(n^{k-1})$ for all $m<k-1$ (if any). Therefore,
$\E(S_n^{2k})-(2k-1)!! n^k
=(2k-1)!!\left\{[(n)_k-n^k]
 +\left(\frac{(k)_2\mu_4}{6\sigma^4}
 +\frac{(k)_3\mu_3^2}{9\sigma^6}\right)(n)_{k-1}\right\}
 +o(n^{k-1})
=O(n^{k-1})
$.
Thus,
$
\E(Z_n^{2k})-\E(Z^{2k})=\frac{\E(S_n^{2k})-(2k-1)!!n^k}{n^k}=\frac{O(n^{k-1})}{n^k}=O(n^{-1})
$.

Combining the above relations, after some algebra, \eqref{eq.lim.odd} and \eqref{eq.lim.even} follow.
\hfill$\square$
\end{proof}

\begin{remark}
\label{rem.rateCLT}
Relations \eqref{eq.lim.odd}, \eqref{eq.lim.even} show that the rate of convergence presented in Theorem \ref{theo.CLT} cannot be improved for $r\ge3$; except, under conditions on the moments of the distribution of $Y$'s. For example, if the distribution of $Y$'s is symmetric, $\E(Z_n^{r})=0$ for each odd integer $r$ with $\E|Y_i|^r<\infty$. Let $5\le r=2k+1\le R$, and assume $\mu_3=0$ and $\mu_5\ne0$. In view of \eqref{eq.lim.odd}, we investigate the second leading coefficients. These coefficients have length $m=k-1$ and correspond to
$\bm{j}_{k-1}^{(1)}=(2,\ldots,2,5)^\ta\in\NN^{k-1}$ and $\bm{j}_{k-1}^{(2)}=(2,\ldots,2,3,4)^\ta\in\NN^{k-1}$. Since $\mu_3=0$, $c_{\bm{j}_{k-1}^{(2)}}(2k+1)E_{\bm{j}_{k-1}^{(2)}}=0$. We compute $c_{\bm{j}_{k-1}^{(1)}}(2k+1)=\frac{(k)_2(2k+1)(2k-1)!!}{30}(n)_{k-1}$ and $E_{\bm{j}_{k-1}^{(1)}}=\frac{\mu_5}{\sigma^5}$. As in proof of Theorem \ref{theo.CLT}, we get $\E(Z_n^{2k+1})=O(n^{-3/2})$; and specifically, $\lim_{n\to\infty}n^{3/2}\E(Z_n^{2k+1})=\frac{(k)_2(2k+1)(2k-1)!!}{30\sigma^5}\mu_5$.
\end{remark}

\begin{remark}
\label{rem.weakCLT}
Let $Y_n$ be iid rv's from $F$ with mean $\mu$, variance $0<\sigma^2<\infty$ and $\E|Y|^a<\infty$ for all $a>0$. Theorem \ref{theo.CLT} shows that $\E(Z_n^r)\to\E(Z^r)$ for all $r=1,2,\ldots,$ where $Z_n$ is the sequence of the standardized averages and $Z$ is a standard normal random variable. Since the distribution of $Z$ is uniquely determined by its moments, it follows that $Z_n\xlongrightarrow{\rm d}Z$ \citep[see, e.g.,][Theo.~2.22, p.~18]{Vaart1998}. This is a weak edition of CLT. For results of this type in a more general format, one can see \citet{Ferger2014}, \citet[pp.~391--392]{Billingsley1995}.
\end{remark}

\subsection{General case}
\label{ssec:rate.general}

Next, we investigate the rate of convergence of the moment convergence of the OLS estimators presented in Theorems \ref{theo.UI1} and \ref{theo.UI2}. In view of Theorem \ref{sec:rate}, we are led to seek a sequence $\alpha_n$ (if it exists) such that $\alpha_n\to\infty$ as $n\to\infty$, for which $\E(\xi_n^r)-\E(\xi^r)=O(\alpha_n^{-1})$ when $r$ is even or $r$ is odd, where $\xi_n$ and $\xi$ are defined as in Theorem \ref{theo.UI1}. Since the canonical CLT is a special case of the linear regression model that satisfies the Assumptions {\rm(iii$'$)} and {\rm(iv)}, the sequence $\alpha_n$ cannot be faster than $n$ for even values of $r$ and than $n^{1/2}$ for odd values of $r$, see Theorem \ref{theo.UI1}. Therefore, for even values of $r$ it is required that $\alpha_n=O(n)$ and for odd values of $r$ that $\alpha_n=O(n^{1/2})$. Note that the sequences $n^a$ for $0<a\le1$ (or $0<a\le1/2$), $\log n$, $n^a\log n$ for $0<a\le1$ (or $0<a\le1/2$), $\log\log n$ etc, that are commonly used in the rate of convergence are as the above. Without loss of generality we assume that $\alpha_n>0$ for all $n$.

Propositions \ref{prop.LM-even}  and \ref{prop.LM-odd} investigate the rate of moment convergence of second and of third moment, respectively, of the OLS estimators.

\begin{proposition}
\label{prop.LM-even}
Let $R > 2$ be fixed. Then, there does not exist a sequence $\alpha_n$ with $\alpha_n\to\infty$ and $\alpha_n=O(n)$, as $n\to\infty$, that determines the rate of the moment convergence for all even moments $r = 2k < R$ and for all models in $\mathcal{LM}(R)$.
\end{proposition}

\begin{proof}
It is sufficient to prove the result for $\alpha_n\to\infty$ and $\alpha_n=o(n)$ as $n\to\infty$. We will prove that such a sequence cannot determine the rate of convergence of the moment convergence not even in the case $r = 2$. First we study the case $0<\alpha_n\nearrow\infty$ and $n^{-1}\alpha_n\searrow0$ as $n\to\infty$. Consider the sequence $\beta_n=n\alpha_n^{-1/2}-(n-1)\alpha_{n-1}^{-1/2}$. Observe that $n\alpha_n^{-1/2}=n\alpha_n^{-1}\alpha_n^{1/2}$ increases to infinity, since both sequences $n\alpha_n^{-1}$, $\alpha_n^{1/2}$ are positive and increase to infinity. Thus, $\beta_n$ is nonnegative [additionally, observe that $0\le\beta_n\le n\alpha_{n-1}^{-1/2}-(n-1)\alpha_{n-1}^{-1/2}=\alpha_{n-1}^{-1/2}\searrow0$ as $n\to\infty$]. Now, consider the model $Y_i=(1+\beta_i)^{1/2}\mu+\varepsilon_i$, $i=1,\ldots,n$, where $\varepsilon_i$'s are iid from $N(0,1)$. Then,
$
\frac{1}{n}\mathbf{X}^\ta \mathbf{X}=\frac{1}{n}\sum_{i=1}^n(1+\beta_i)=\frac{n+n\alpha_n^{1/2}}{n}=1+\alpha_n^{-1/2}\searrow1
$
as $n\to\infty$. Thus, this model belongs in $\mathcal{LM}(R)$ for each $R>2$. Applying \eqref{eq.N_p}, $\xi_n=\sqrt{n}(\h{\mu}-\mu)\xlongrightarrow{\rm d}\xi\law N(0,1)$, where $\h{\mu}$ is the OLS estimator of $\mu$. Writing $\xi_n=\frac{1}{n^{1/2}\left(1+\alpha_n^{-1/2}\right)}\sum_{i=1}^n(1+\beta_i)^{1/2}\varepsilon_i$, we get $\E(\xi_n^2)=\Var(\xi_n)=\frac{1}{n\left(1+\alpha_n^{-1/2}\right)^2}\sum_{i=1}^n(1+\beta_i) =\frac{1}{1+\alpha_n^{-1/2}}=-\frac{\alpha_n^{-1/2}}{1+\alpha_n^{-1/2}}$. Hence,
$
\alpha_n\big(\E(\xi_n^2)-\E(\xi^2)\big)=-\frac{\alpha_n^{1/2}}{1+\alpha_n^{-1/2}}\to-\infty
$
as $n\to\infty$.

Let now $\alpha_n\to\infty$ and $n^{-1}\alpha_n\to0$ as $n\to\infty$. Then, there exists a subsequence $\alpha_{k_n}$ of $\alpha_{n}$ such that $0<\alpha_{k_n}\nearrow\infty$. The subsequence $k_n^{-1}\alpha_{k_n}$ of $n^{-1}\alpha_{n}$ has positive terms and $k_n^{-1}\alpha_{k_n}\to0$. Thus, there exists another subsequence $l_{k_n}^{-1}\alpha_{l_{k_n}}$ of $k_n^{-1}\alpha_{k_n}$ such that $l_{k_n}^{-1}\alpha_{l_{k_n}}\searrow0$ (obviously, $\alpha_{l_{k_n}}\nearrow\infty$). Set $l_{k_0}=0$ with $\alpha_{l_{k_0}}=0$, and define the sequence $\tilde{\alpha_i}$, $i=1,2,\ldots,$ as
$
\tilde{\alpha_i}\doteq
\left\{\begin{array}{@{\hspace{0ex}}c@{,}c@{\ \textrm{for} \ }l}
  \alpha_i, && i=l_{k_n}\\
  \frac{\alpha_{l_{k_{n+1}}}-\alpha_{l_{k_n}}}{l_{k_{n+1}}-l_{k_n}}i
     +\frac{\alpha_{l_{k_{n}}}l_{k_{n+1}}-\alpha_{l_{k_{n+1}}}l_{k_n}}{l_{k_{n+1}}-l_{k_n}}
      && l_{k_n}<i<l_{k_{n+1}} \ \textrm{(if any)}
\end{array}
\right..
$
By the construction of $\tilde{\alpha_n}$, $0<\tilde{\alpha}_n\nearrow\infty$ and $n^{-1}\tilde{\alpha}_n\searrow0$. Assume the model $Y_i=(1+\beta_i)^{1/2}\mu+\varepsilon_i$, $i=1,\ldots,n$, where $\beta_n=n\tilde{\alpha}_n^{-1/2}-(n-1)\tilde{\alpha}_{n-1}^{-1/2}$ and $\varepsilon_i$'s are iid from $N(0,1)$. From the above analysis we get
$
\alpha_{l_{k_n}}\left(\E\big(\xi_{l_{k_n}}^2\big)-\E(\xi^2)\right)
=\tilde{\alpha}_{l_{k_n}}\left(\E\big(\xi_{l_{k_n}}^2\big)-\E(\xi^2)\right)\to-\infty
$
as $n\to\infty$, completing the proof.
\hfill$\square$
\end{proof}

For the case of the third moment first we present the following example.

\begin{example}
\label{exm.odd}
Let $x_n$ be a convergent sequence of real numbers, $x_n\to c\ne0$ as $n\to\infty$. Consider the model $Y_i=x_i\mu+\varepsilon_i$, where $\varepsilon_i$'s are iid with $\E(\varepsilon_i)=0$, $\Var(\varepsilon_i)=\sigma^2$, $\E|\varepsilon_i|^3<\infty$ and $\E(\varepsilon_i^3)=\mu_3$. Then,
$
\frac{1}{n}\mathbf{X}^\ta\mathbf{X}=\frac{1}{n}\sum_{i=1}^n x_i^2\to c^2
$
as $\to\infty$, since $x_n^2\to c^2$. Applying \eqref{eq.N_p}, $\xi_n=\sqrt{n}(\h{\mu}-\mu)\xlongrightarrow{\rm d}\xi\law N(0,c^2\sigma^2)$, where $\h{\mu}$ is the OLS estimator of $\mu$. Writing $\xi_n=\frac{n^{1/2}}{\mathbf{X}^\ta\mathbf{X}}\sum_{i=1}^n x_i\varepsilon_i$, we get $\E(\xi_n^3)=\frac{n^{3/2}\sum_{i=1}^n x_i^3}{(\mathbf{X}^\ta\mathbf{X})^3}\mu_3$. Thus,
$
\sqrt{n}\E(\xi_n^3)=\frac{n^{-1}\sum_{i=1}^n x_i^3}{(n^{-1}\mathbf{X}^\ta\mathbf{X})^3}\mu_3\to \frac{c^3}{(c^2)^3}\mu_3=\frac{\mu_3}{c^3}
$
as $n\to\infty$, since $x_n^3\to c^3$. Therefore, for each model such that the rate of convergence of the third moment is the same as in the canonical CLT.
\end{example}

In view of Example \ref{exm.odd}, one may incorrectly suggest that, in general, the moment convergence of the third order (or that of the odd moments) of Theorem \ref{theo.UI1} has the rate $n^{1/2}$. Proposition \ref{prop.LM-odd} shows that this conclusion is false.

\begin{proposition}
\label{prop.LM-odd}
Let $R$ be fixed, $R>3$. Then, there does not exist $a\in(0,1/2)$ such that the sequence $n^a$ determines the rate of the moment convergence for all odd moments $r=2k+1<R$ and for all models in $\mathcal{LM}(R)$.
\end{proposition}

\begin{proof}
Let $a\in(0,1/2)$ be fixed. We will prove that such the sequence $n^a$ cannot determine the rate of convergence of the moment convergence not even in the case $r=3$. Set $b=(1-2a)/a>0$ and consider the model $Y_i=x_i\mu+\varepsilon_i$, $i=1,\ldots,n$, where $x_i=k^{b/2}\mathds{1}_{\left\{i=\lfloor k^{b+1}\rfloor, k\in\NN\right\}}$ ($\mathds{1}$ stands for the indicator function) and $\varepsilon_i$'s are iid with $\E(\varepsilon_i)=0$, $\Var(\varepsilon_i)=1$, $\E|\varepsilon_i|^3<\infty$ and $\E(\varepsilon_i^3)=\mu_3\ne0$. By construction, $\mathbf{X}^\ta\mathbf{X} =\sum_{k=1}^{\lfloor n^{1/(b+1)}\rfloor}k^b$; and thus, $\mathbf{X}^\ta\mathbf{X}\approxn \int_{0}^{\lfloor n^{1/(b+1)}\rfloor}t^b\du t=\frac{\lfloor n^{1/(b+1)} \rfloor^{b+1}}{b+1}\approxn \frac{n}{b+1}$; that is, {\rm(iii$'$)} holds with $V^{-1}=(b+1)^{-1}$. Applying \eqref{eq.N_p}, $\xi_n=\sqrt{n}(\h{\mu}-\mu)\xlongrightarrow{\rm d}\xi\law N(0,b+1)$, where $\h{\mu}$ is the OLS estimator of $\mu$. The third moment of $\xi_n$ is $\E(\xi_n^3)=\frac{n^{3/2}\sum_{i=1}^n x_i^3}{(\mathbf{X}^\ta\mathbf{X})^3}\mu_3$. Now, $\sum_{i=1}^n x_i^3=\sum_{k=1}^{\lfloor n^{1/(b+1)}\rfloor}k^{3b/2}=\int_{0}^{\lfloor n^{1/(b+1)}\rfloor}t^{3b/2}\du t=\frac{2\lfloor n^{1/(b+1)} \rfloor^\frac{3b+2}{2}}{3b+2}\approxn \frac{2n^\frac{3b+2}{2b+2}}{3b+2}$. By substitution $b=\frac{1-2a}{a}$, $n^a\E(\xi_n^3)\approxn \frac{2(1-a)^3\mu_3}{a^2(3-4a)}n^{\frac{a(1-2a)}{2(1-a)}}$.
Therefore, since $\frac{2(1-a)^3\mu_3}{a^2(3-4a)}\ne0$ and $\frac{a(1-2a)}{2(1-a)}>0$, $n^a\E(\xi_n^3)\ne O(1)$, completing the proof.
\hfill$\square$
\end{proof}

\section{Discussion}
\label{sec.disc}

In this paper we first prove the uniform integrability of the sequence $|\xi_n|^r$, where $\xi_n$ is defined by \eqref{eq.N-xi}. The uniform integrability of this sequence of the OLS estimators guarantees the convergence of the moments of the rv's $\xi_n$ to the moments of the limiting rv $\xi$. Next, we generalize this result for a sequence of random vectors with components defined as linear combinations of the components of the OLS estimators of a linear model.

We also present a sharp improvement of \citeauthor{Bahr1965}'s theorem that pertains to the rate of convergence of the moments of $Z_n$, the sequence of the standardized averages, to the corresponding moments of the limiting random variable $Z$. This case presents the simplest linear model where $Y_i=\mu+\varepsilon_i$, $i=1,2\ldots,n$. A difficult and open problem is the determination of the impact of the design matrix $\mathbf{X}$ on the rate of the moment convergence of $\xi_n$. In the aforementioned simplest case we showed that the rate of convergence of the moment sequences depends on whether the moment power $r$ is even or odd. In view of Propositions \ref{prop.LM-even} and \ref{prop.LM-odd} it is clear that we can find a sequence of design matrices satisfying {\rm(iii$'$)} such that the rate of convergence of the second moment of $\xi_n$ of Theorem \ref{theo.UI1} is arbitrarily slow. Further, we can find a sequence of design matrices satisfying {\rm(iii$'$)} such that the rate of convergence of the third moment of $\xi_n$ is slower than $n^a$, for any given $a\in(0,1/2)$. Furthermore, we indicated the practical applicability of these results in, for example, addressing problems in cross validation.


\end{document}